\documentclass[12pt,fullpage,leqno]{amsart}
\usepackage{amsmath,amssymb,amsthm,amsfonts,amscd}
\usepackage{xr}
\usepackage{graphicx}
\usepackage{epsfig}
\usepackage{color}
\numberwithin{equation}{section}
\newtheorem{main}{Theorem}[section]
\newtheorem{bounded0}[main]{Theorem}
\newtheorem{total}[main]{Theorem}
\newtheorem{Simon}{Lemma}[section]

\newtheorem{Lp}[Simon]{Proposition}
\newtheorem{Lp2}[Simon]{Propostion}

\newtheorem{MeanValue}{Lemma}[section]

\newtheorem{Sobolev}[MeanValue]{Lemma}
\newtheorem{estimate}[MeanValue]{Theorem}

\newtheorem{theorem}{Theorem}[section]
\newtheorem{bound}[theorem]{Proposition}

\newtheorem{pro1}[theorem]{Theorem}
\newtheorem{pro2}[theorem]{Theorem}
\newtheorem{2}{Remark}
\newtheorem{3}[2]{Remark}
\newtheorem{4}[2]{Remark}

\begin{document}
\title{On stable hypersurfaces with constant mean curvature in Euclidean spaces}

\author{Jinpeng Lu}
\address{Department of Mathematical Sciences, Tsinghua University,
Beijing 100084, P.R. CHINA} \email{lujp11@mails.tsinghua.edu.cn}

\maketitle
\begin{abstract}
In this paper, we derive curvature estimates for strongly stable hypersurfaces with constant mean curvature immersed in $\mathbb{R}^{n+1}$, which show that the locally controlled volume growth yields a globally controlled volume growth if $\partial M=\emptyset$. Moreover, we deduce a Bernstein-type theorem for complete stable hypersurfaces with constant mean curvature of arbitrary dimension, given a finite $L^p$-norm curvature condition.
\end{abstract}

\section{Introduction}

\indent In \cite{SSY}, R.~Schoen, L.~Simon and S.T.~Yau proved a curvature type estimate for stable minimal hypersurfaces in higher dimensional Riemannian manifolds, which yielded the general Bernstein theorem for $n\leqslant 5$. Then for the surface case, the estimate of this type was generalized to stable minimal surfaces immersed in a general three-manifold by R.~Schoen in \cite{RS}, and he proved that a complete stable minimal surface in a three-manifold of nonnegative Ricci curvature is totally geodesic. Later in \cite{CM}, T.~Colding and W.~Minicozzi gave another type of estimates for stable minimal surfaces in a three-manifold with trivial normal bundle.\\
\indent As to hypersurfaces with constant mean curvature (or simply CMC hypersurfaces), P.~B$\acute{\textrm{e}}$rard and L.~Hauswirth obtained a curvature type estimate for strongly stable CMC surfaces immersed in space forms in \cite{BEHA}, which was generalized in \cite{SZ} to strongly stable CMC surfaces in a general three-manifold with bounded sectional curvature and trivial normal bundle by S.~Zhang.\\
\indent For stable CMC hypersurfaces in higher dimension, Y.~Shen and X.~Zhu have proved that a complete stable minimal hypersurface in $\mathbb{R}^{n+1}$ with $\int_M |A|^n dv<\infty$ must be a hyperplane \cite{ShZh1}, which was later generalized to strongly stable CMC hypersurfaces with $\int_M |\phi|^n dv<\infty$ by themselves \cite{ShZh}. For conditions of this type, the conclusion remains true, if $\int_M |\phi|^2 dv<\infty$ $(n\leqslant 5)$ by Alencar and do Carmo \cite{AD2}, or $\int_M |\phi| dv<\infty$ $(n\leqslant 6)$ by Shen and Zhu \cite{ShZh}.\\
\indent In this paper, we will prove curvature estimates for strongly stable CMC hypersurfaces in higher dimensional Euclidean spaces, given a local volume growth condition. We also prove a Bernstein-type theorem without a dimensional assumption, if $\int_M |\phi|^p dv<\infty$, for some $p\in[1-\frac{2}{n},\infty)$.\\

\indent Throughout this paper, $\phi: M^n \to N^{n+1}(c)$ is an isometric immersion with constant mean curvature $H$ into a space form of constant sectional curvature $c$. The second fundamental form is denoted by $A$, and $\phi$ is the traceless second fundamental form defined by $$\phi=A-HId.$$ The volume element of $M$ is denoted by $dv$. $B_R$ denotes the intrinsic geodesic ball in $M$, and the extrinsic ball is denoted by $$D_R(x)=\{y\in M: r(x,y)< R\}=\overline B_R(x)\cap M,$$
where $\overline B_R(x)$ is the open ball of radius $R$ in $\mathbb{R}^{n+1}$, and $r$ is the Euclidean distance function in $\mathbb{R}^{n+1}$. $\overline B^C_R(x)$ and $D^C_R(x)$ stand for the connected component of $\overline B_R(x)$ and $D_R(x)$ containing $x$. Note that $B_R(x)$ is contained in $D^C_R(x)$.\\
\indent We say a CMC hypersurface is strongly stable, if for any $h\in C_0^{\infty}(M)$,
$$\int_M (|A|^2+nc)h^2\leqslant\int_M |\nabla h|^2,$$
which is equivalent to
\begin{equation}
\int_M [|\phi|^2+n(H^2+c)]h^2\leqslant\int_M |\nabla h|^2.
\end{equation}
On the other hand, a hypersurface with constant mean curvature is weakly stable, if for all $h\in C_0^{\infty}(M)$ satisfying $\int_M f=0$,
$$\int_M [|\phi|^2+n(H^2+c)]h^2\leqslant\int_M |\nabla h|^2.$$
For $H=0$, an immersion is called stable when it refers to the natural stability on minimal hypersurface, i.e., for all $h\in C_0^{\infty}(M)$.\\

\indent The main results of this paper are as follows.
\begin{main}
Let $M^n$ be an oriented strongly stable hypersurface with constant mean curvature $H$ immersed in $\mathbb{R}^{n+1}$ $($$n\geqslant 2$$)$. If there exists $R_0$, such that for all $R<\min\{R_c, r(x,\partial M)\}$ and for some $\alpha$ $($not necessarily positive$)$,
\begin{equation}
|D^C_{R}(x)|\leqslant C_1 R^{n+\alpha}, \\
\end{equation}
where the constant $C_1$ is independent on $R$ and $x$. Then for some $\delta=\delta(\alpha)>0$, there exists $r_0=r_0(n,H,\alpha,\delta)\leqslant 1$ such that for any $x\in M$ satisfying $\overline{B}^C_{r_0}(x)\cap\partial M =\emptyset$, the following inequality holds
$$\max_{0\leqslant\sigma\leqslant r_0}\sigma^{2+2\delta}\sup_{D^C_{r_0-\sigma}(x)}|\phi|^2\leqslant 1.$$
\end{main}

This theorem shows that the controlled local volume growth implies a controlled global volume growth.

\begin{bounded0}
Suppose $M^n$ is an oriented strongly stable CMC hypersurface immersed in $\mathbb{R}^{n+1}$ $($$n\geqslant 2$$)$. If $M$ satisfies the volume condition $(1.2)$, then there exists a constant $C$, such that
$$\sup_{\{r(x,\partial M)>\epsilon\}} |\phi|(x)\leqslant C(n,H,\epsilon).$$
In particular, if $\partial M=\emptyset$, then $|\phi|$ is bounded on $M$. So there exists $K=K(n,H,\sup\limits_{M}|\phi|)$, such that for any $R$,
$$|B_R|\leqslant V(n,K,R),$$
where $V(n,K,R)$ stands for the volume of the geodesic ball of radius $R$ in the n-dimensional space form of sectional curvature K.
\end{bounded0}

As an application of the theorems above, we prove the following Bernstein-type theorem.

\begin{total}
Let $M^n$ be a complete oriented hypersurface with constant mean curvature $H$ immersed in $\mathbb{R}^{n+1}$ $($$n\geqslant 2$$)$ and $\partial M=\emptyset$. Assume there exists
$$1-\frac{2}{n}\leqslant p<\infty,$$
such that $$\int_M |\phi|^p dv<\infty.$$
$(1)$ If $M$ is weakly stable with $H\neq 0$, then $M$ has to be a round sphere;\\
$(2)$ If $M$ is non-compact, then $M$ has to be minimal;\\
$(3)$ If $M$ is stable minimal and $1-\frac{2}{n}\leqslant p\leqslant n$, then $M$ has to be a hyperplane.
\end{total}

\indent This paper is organized as follows. In section 2, we follow \cite{SSY} to get $L^p$ estimates for $|\phi|$ in space forms, using Simons' inequality and stability inequality. Section 3 is devoted to obtaining a local curvature estimate by a general mean value equality, which depends on a locally controlled volume growth. In Section 4, we prove a Bernstein-type theorem, given a finite $L^p$-norm curvature condition.\\

\section{$L^p$ Estimates}

At the very beginning, we need a Simons' inequality. See \cite{CY} or \cite{AD} for reference.

\begin{Simon}
\begin{equation}
|\phi|\Delta|\phi|+|\phi|^4\geqslant\frac{2}{n}|\nabla|\phi||^2+n(H^2+c)|\phi|^2-\frac{n(n-2)}{\sqrt{n(n-1)}}H|\phi|^3.\\
\end{equation}
\end{Simon}

Then, we can derive a standard $L^p$ estimate for $|\phi|$.

\begin{Lp}
Suppose $M^n$ is an oriented strongly stable hypersurface with constant mean curvature $H$ immersed in a space form $N^{n+1}(c)$ of constant sectional curvature $c$. Then for any $0\leqslant q< \sqrt{\frac{2}{n}}$, we have
$$\int_M f^{2q+4}|\phi|^{2q+4}\leqslant\beta(n,H,q,c)(\int_M f^{2q+4}+\int_M |\nabla f|^{2q+4}),$$
for any $f\in C_0^{\infty}(M)$.
\end{Lp}

\begin{proof}
First, we multiply (2.1) by $f^2|\phi|^{2q}$, integrate over $M$ and integrate by parts, which gives
\begin{eqnarray}
&\frac{2}{n}\int_M f^2|\phi|^{2q}|\nabla|\phi||^2&+\int_M n(H^2+c)f^2|\phi|^{2q+2}\nonumber \\
&-\int_M \frac{n(n-2)}{\sqrt{n(n-1)}}Hf^2|\phi|^{2q+3}&\leqslant \int_M f^2|\phi|^{2q+4}-\int_M 2f|\phi|^{2q+1}\nabla|\phi|\cdot\nabla f \nonumber \\
& &-(2q+1)\int_M f^2|\phi|^{2q}|\nabla|\phi||^2.
\end{eqnarray}
On the other hand, replacing $f$ in (1.1) by $f|\phi|^{q+1}$, we have
\begin{eqnarray}
\int_M f^2|\phi|^{2q+4}&+&\int_M n(H^2+c) f^2|\phi|^{2q+2}\leqslant\int_M |\nabla f|\phi|^{q+1}|^2 \nonumber \\
                                                    &= & \int_M |\nabla f|\phi|^{q+1}+(q+1)f|\phi|^q\nabla|\phi||^2 \nonumber \\
                                                    &=& \int_M |\nabla f|^2|\phi|^{2q+2}+(q+1)^2\int_M f^2|\phi|^{2q}|\nabla|\phi||^2\nonumber \\
                                                    & & 2(q+1)\int_M f|\phi|^{2q+1}\nabla|\phi|\cdot\nabla f.
\end{eqnarray}
Then, by multiplying (2.2) by $(1+q)$ and adding up with (2.3), we get
\begin{eqnarray}
(\frac{2}{n}+q)(1+q)\int_M f^2|\phi|^{2q}|\nabla|\phi||^2 \leqslant q\int_M f^2|\phi|^{2q+4}+\int_M|\nabla f|^2|\phi|^{2q+2} \nonumber \\
+\frac{n(n-2)}{\sqrt{n(n-1)}}H(1+q)\int_M f^2 |\phi|^{2q+3}-n(H^2+c)(q+2)\int_M f^2|\phi|^{2q+2}.
\end{eqnarray}
Using Young's inequality on the last term of (2.3) and using (2.4), we obtain
\begin{eqnarray}
[1-\frac{(1+q+\varepsilon)q}{q+\frac{2}{n}}]\int_M f^2|\phi|^{2q+4} \leqslant (n-2)H\beta_1(n,q,\varepsilon)\int_M f^2 |\phi|^{2q+3} \nonumber \\
+\beta_2(n,q,\varepsilon)\int_M |\phi|^{2q+2}|\nabla f|^2-\beta_3(n,q,\varepsilon)(H^2+c)\int_M f^2 |\phi|^{2q+2},
\end{eqnarray}
where the constants $\beta_1,\beta_2,\beta_3>0$. To ensure
$$1-\frac{(1+q+\varepsilon)q}{q+\frac{2}{n}}\geqslant 0,$$
we take $q^2<\frac{2}{n}$ and fix a proper $\varepsilon$. Again from (2.5) by Young's inequality, we obtain\\
\begin{equation}
\int_M f^2|\phi|^{2q+4}\leqslant \beta_6(n,H,q,c)(\int_M f^2+\int_{M^+}\frac{|\nabla f|^{2q+4}}{f^{2q+2}}).
\end{equation}
So replacing $f$ in (2.6) by $f^{q+2}$, the proof is done.
\end{proof}

Now we consider the case $c=0$, i.e. $N=\mathbb{R}^{n+1}$. From Lemma 2.1, we know that in the $\mathbb{R}^{n+1}$ case,
\begin{eqnarray}
\frac{1}{2}\Delta |\phi|^2 &\geqslant &(1+\frac{2}{n})|\nabla|\phi||^2-|\phi|^4+nH^2|\phi|^2-\frac{n(n-2)}{\sqrt{n(n-1)}}H|\phi|^3 \nonumber \\
                           & \geqslant &(1+\frac{2}{n})|\nabla|\phi||^2 -|\phi|^4+nH^2|\phi|^2-\frac{(n-2)^2}{4(n-1)}|\phi|^4-nH^2|\phi|^2 \nonumber \\
                           & \geqslant & (1+\frac{2}{n})|\nabla|\phi||^2-\frac{n^2}{4(n-1)}|\phi|^4.
\end{eqnarray}
Thus,
\begin{equation}
|\phi|\Delta|\phi|+\frac{n^2}{4(n-1)}|\phi|^4\geqslant \frac{2}{n}|\nabla|\phi||^2.
\end{equation}
Then by following the proof of Proposition 2.2, we get the following proposition.

\begin{Lp2}
Suppose $M^n$ is an oriented strongly stable hypersurface with constant mean curvature $H$ immersed in $\mathbb{R}^{n+1}$. If $n\leqslant 4$, and there exists $q\geqslant 0$ satisfying
\begin{equation}
\frac{n^2}{4(n-1)}q^2+\frac{(n-2)^2}{2(n-1)}q+[\frac{n^2}{4(n-1)}-\frac{2}{n}-1]<0.
\end{equation}
Then for any $f\in C_0^{\infty}(M)$,
$$\int_M f^{2q+4}|\phi|^{2q+4}\leqslant\beta(n,q)\int_M |\nabla f|^{2q+4}.$$
\end{Lp2}

\section{Curvature Estimates}
In this section, we will first prove Theorem 1.1. This type of curvature estimate can be seen in \cite{CHSC} or \cite{CM}.

\begin{MeanValue}
Let $M^n$ be an oriented hypersurface with constant mean curvature $H$ immersed in $\mathbb{R}^{n+1}$ $($$n\geqslant 2$$)$. Assume $x_0\in M$ and $R\leqslant 1$ satisfying $\overline{B}^C_R(x_0)\cap \partial M =\emptyset$. If $f$ is a nonnegative function with $\Delta f \geqslant -C_2R^{-(2+2\delta)}f ,\;\delta>0$, then
$$f^2(x_0)\leqslant \frac{C(n,H)}{R^{n-\delta}}\int_{D^C_R(x_0)}f^2 dv.$$
\end{MeanValue}

\begin{proof}
It is obvious that we only need to consider the case that $\overline{B}_R(x_0)$ is connected. Let $\overline M =M\times[-R^{1+\delta},R^{1+\delta}]$, so $\overline M \to \mathbb{R}^{n+2}$ is also a hypersurface with constant mean curvature under the standard metric of a product space.\\
\indent Define a function on $\overline M $ by
\begin{equation}
g(x,t)=f(x)\exp{\frac{\sqrt{C_2}t}{R^{1+\delta}}},
\end{equation}
and then $$\Delta_{\overline M}g(x,t)=\exp{\frac{\sqrt{C_2}t}{R^{1+\delta}}}(\Delta f+\frac{C_2}{R^{2+2\delta}})\geqslant 0.$$
\indent Thus, by the mean value inequality for subharmonic functions proved by Michael and Simon \cite{MS}, letting $x=x_0,\;t=0$ and using the condition $R\leqslant 1$, we get
\begin{equation}
f^2(x_0)\leqslant \big(1+\frac{H}{1!}+\cdots+\frac{H^{n+1}}{(n+1)!}\big)\frac{1}{\omega_{n+1}R^{n+1}}\int_{\overline{D_R}(x_0,0)\subset \overline M}f^2 d\overline{v},
\end{equation}
where $\omega_{n+1}$ denotes the volume of the unit ball in $\mathbb{R}^{n+1}$, and $d\overline{v}$ denotes the volume element of $\overline M$.\\
By
$$\{\overline{D_R}(x_0,0)\subset \overline M\}\subset \{D_R(x_0)\subset M\}\times[-R^{1+\delta},R^{1+\delta}].$$
Then, we have
\begin{equation}
f^2(x_0)\leqslant \frac{C(n,H)}{R^{n-\delta}}\int_{D_R(x_0)}f^2 dv.
\end{equation}
\end{proof}

\indent By the general mean value inequality in Lemma 3.1, we are now able to prove Theorem 1.1.
\begin{proof}[Proof of Theorem 1.1]
Without loss of generality, we assume $R_c=1$.\\
For a chosen $0<r\leqslant 1$, and $x\in M$ satisfying $\overline{B}^C_{r}(x)\cap\partial M =\emptyset$, let $\sigma_0\in[0,r]$,
such that $$\max_{0\leqslant\sigma\leqslant r}\sigma^{2+2\delta}\sup_{D^C_{r-\sigma}(x)}|\phi|^2=\sigma_0^{2+2\delta}\sup_{D^C_{r-\sigma_0}(x)}|\phi|^2.$$
Let $x_0\in \overline{D}^C_{r-\sigma_0}(x)$, such that $$\sup_{\overline{D}^C_{r-\sigma_0}(x)}|\phi|^2=|\phi|^2(x_0).$$
Then $$\sup_{D^C_{\frac{\sigma_0}{2}}(x_0)}|\phi|^2\leqslant\sup_{D^C_{r-\frac{\sigma_0}{2}}(x)}|\phi|^2\leqslant2^{2+2\delta}|\phi|^2(x_0).$$
If $\sigma_0^{2+2\delta}|\phi|^2(x_0)\leqslant 1$, then the inequality is true. Otherwise, assume $$\sigma_0^{2+2\delta}|\phi|^2(x_0)>1.$$
Choose $2\eta<\sigma_0\leqslant r<1$, such that $(2\eta)^{2+2\delta}|\phi|^2(x_0)=1$. Then we get
$$\sup_{D^C_{\eta}(x_0)}\eta^{2+2\delta}|\phi|^2 \leqslant \sup_{D^C_{\frac{\sigma_0}{2}}(x_0)}|\phi|^2\eta^{2+2\delta}\leqslant
2^{2+2\delta}|\phi|^2(x_0)\eta^{2+2\delta}=1.$$
So from (2.7), we know that in $D^C_{\eta}(x_0)$,
\begin{equation}
\Delta |\phi|^2\geqslant -C_3(n)|\phi|^4 \geqslant -C_3(n)\eta^{-(2+2\delta)}|\phi|^2.
\end{equation}
By using Lemma 3.1 and Proposition 2.2 and choosing a standard cut-off function, we have
\begin{eqnarray*}
|\phi|^4(x_0)&\leqslant & \frac{C(n,H)}{\eta^{n-\delta}}\int_{D^C_{\eta}(x_0)}|\phi|^4 dv \\
             &\leqslant & \frac{C_4(n,H)}{\eta^{n-\delta}} (1+\frac{1}{\eta^{4}})|D^C_{2\eta}(x_0)|\\
             &\leqslant & C_5(n,H)\eta^{\alpha+\delta-4}.
\end{eqnarray*}
Using $|\phi|^2(x_0)=(2\eta)^{-(2+2\delta)}$, then
\begin{equation}
\eta^{-5\delta-\alpha}\leqslant C_5(n,H)2^{4+4\delta}.
\end{equation}
So we choose some $\delta>-\frac{\alpha}{5}$, and then choose $r$ to be
\begin{equation}
0<r=r_0<C_5(n,H)^{-\frac{1}{5\delta+\alpha}}2^{-\frac{4+4\delta}{5\delta+\alpha}}.
\end{equation}
Then we get a contradiction, which completes the proof.
\end{proof}

\begin{2}
Note that condition $(1.2)$ is equivalent to the following condition:\\
If there exists $R_c$, such that for any $R<\min\{R_c,r(x,\partial M)\}$,
\begin{equation}
|D^C_{R}(x)|\leqslant C_1, \\
\end{equation}
where the constant $C_1$ is independent on $R$ and $x$.
\end{2}

\begin{proof}[Proof of Theorem 1.2]
If $M$ satisfies the volume condition (1.2), then from Theorem 1.1, we set $r=\min\{r_0,r(x,\partial M)\}$ and $\sigma=\frac{r}{2}$, obtaining that for any $x\in M$,
\begin{equation}
|\phi(x)|^2\leqslant \sigma^{-(2+2\delta)}\leqslant 2^{2+2\delta}r^{-(2+2\delta)},
\end{equation}
which completes the proof by the volume comparison theorem.
\end{proof}

\begin{3}
When $\partial M=\emptyset$, it is remarkable that the upper bound of $|\phi|$ only depends on $n$ and $H$, not even the hypersurface $M$ itself. Theorem $1.2$ shows that the locally controlled volume growth $(1.2)$ for extrinsic balls implies a globally controlled volume growth for intrinsic balls.\\
The boundary point set $\partial M$ is defined as the limit point set of Cauchy series of points in $M$ under the Euclidean topology. But our proof is still valid even for hypersurfaces like $\sqrt{x^2+y^2}=\sin\frac{1}{z}+1$ $(z>0)$, which has topological boundary points. Therefore, the condition $\partial M=\emptyset$, as well as in the following theorems, can be relaxed.
\end{3}

\indent Assume $|D^C_{R}|R^{-n-\alpha}$ is bounded for any radius $R$ and some $\alpha$, which clearly satisfies the volume condition (1.2). So from Theorem 1.2 and (2.7), we get
\begin{equation}
\Delta f+C_6(n,H)f\geqslant 0,
\end{equation}
where $$f=|\phi|^2.$$
\indent With the Sobolev inequality proved in \cite{HS} or \cite{MS}, we are able to prove the Schoen-Simon-Yau type curvature estimate by Moser iteration.

\begin{Sobolev}[Sobolev Inequality]
Suppose $M^n\to \mathbb{R}^{n+1}$ is an isometric immersion with constant mean curvature $H$. Let $h\in C_0^1(M)$ be nonnegative, then
$$\big(\int_M h^{\frac{n}{n-1}}dv\big)^{\frac{n-1}{n}}\leqslant C(n)\int_M (|\nabla h|+h|H|)dv.$$
\end{Sobolev}

\begin{estimate}
Let $M^n$ be an oriented strongly stable CMC hypersurface immersed in $\mathbb{R}^{n+1}$ $($$n\geqslant 2$$)$ with $\partial M=\emptyset$, and assume $|D^C_{R}(x)|\leqslant \kappa R^{n+\alpha}$ for any $R$ and some $\alpha$.\\
$(1)$ If $n\geqslant 2$, and $R<1$, then
$$\sup_{B_{\theta R}(x)}|\phi|\leqslant C(n,H,\theta,\kappa)R^{\frac{\alpha-n}{4}-1},$$
for every $\theta\in(0,1)$;\\
$(2)$ If $2\leqslant n\leqslant 4$, and $R\geqslant 1$, then
$$\sup_{B_{\theta}(x)}|\phi|\leqslant C(n,H,\theta,\kappa)R^{\frac{\alpha+n}{2q+4}-1},$$
for every $\theta\in(0,1)$ and for $0\leqslant q\leqslant 0.18$.\\
Here the inequalities are still true for $D^C_R(x)$.
\end{estimate}

\begin{proof}
From (3.9), we know that, for $s\geqslant 1$,
\begin{equation}
\Delta f^{s}+C_6(n,H)s f^{s}\geqslant 0.
\end{equation}
For $n\geqslant 2$, we multiply (3.9) by $\xi^2 f^{s}$ and replace $h$ by $\xi^2 f^{2s}$ in Sobolev inequality in Lemma 3.2. Then by combining the two inequalities together and choosing $\xi$ to be standard cut-off functions defined on intrinsic geodesic balls of $M$, we get
\begin{equation}
\big(\int_{B_{R_{i+1}}(x)} f^{\frac{2n}{n-1}s}dv\big)^{\frac{n-1}{n}}\leqslant C_7(n,H)s[\frac{1}{(R_i-R_{i+1})^2} +1] \int_{B_{R_i}(x)} f^s.
\end{equation}
So by taking $R_0<1,\;R_{i+1}=R_i-r_i,\;r_i=2^{-i-2}R_0<1,\;q_i=(\frac{n}{n-1})^i$, we obtain
\begin{equation}
I_{i+1}\leqslant C_7(n,H)^{\frac{1}{q_i}}q_i^{\frac{1}{q_i}}[r_i^{-2}+1]^{\frac{1}{q_i}} I_i \leqslant C_7(n,H)^{\frac{1}{q_i}}q_i^{\frac{1}{q_i}}(2r_i^{-2})^{\frac{1}{q_i}} I_i,
\end{equation}
where $$I_i=\big(\int_{B_{R_i}(x)}f^{2q_i}\big)^{\frac{1}{q_i}}.$$
So from Proposition 2.2, and for $R\geqslant R_0$, we have
\begin{eqnarray}
\sup_{B_{\frac{R_0}{2}}(x)}|\phi|^4 &=& \lim_{k\to\infty} I_k \nonumber \\
                               &\leqslant& C_{8}(n,H) R_0^{-2\sum\limits_{k=0}^{\infty}\frac{1}{q_k}} \int_{B_{R_0}(x)}|\phi|^4 \nonumber \\
                               &\leqslant& C_{9}(n,H) R_0^{-2n} [1+(2R-R_0)^{-4}]|D_{2R}(x)|.
\end{eqnarray}
Take $R=R_0<1$, and then we get
\begin{equation}
\sup_{B_{\frac{R}{2}}(x)}|\phi|\leqslant C_{10}(n,H,\kappa)R^{\frac{\alpha-n}{4}-1}.
\end{equation}
(2) By setting $q_i=(q+2)(\frac{n}{n-1})^i$ ($q\geqslant 0$) and $R>R_0=1$, we can derive an inequality similar to (3.13), which completes the proof.
\end{proof}

\section{Hypersurfaces With Finite $L^p$-norm Curvature}

In this section, we will prove Theorem 1.3.\\
Rewriting (2.7) by using $f=|\phi|^2$, we get
\begin{equation}
\Delta f-(1+\frac{2}{n})\frac{|\nabla f|^2}{2f}+\frac{n^2}{2(n-1)}f^2\geqslant 0.
\end{equation}
Consider $f^{\gamma}$, where $\gamma\geqslant \frac{1}{2}-\frac{1}{n}$.
\begin{eqnarray}
\Delta f^{\gamma}+\frac{\gamma n^2}{2(n-1)}f^{\gamma+1}&=& \gamma(\gamma-1)f^{\gamma-2}|\nabla f|^2+\gamma f^{\gamma-1}\Delta f+\frac{\gamma n^2}{2(n-1)}f^{\gamma+1} \nonumber \\
&=& \gamma f^{\gamma-1}[\Delta f-(1-\gamma)\frac{|\nabla f|^2}{f}+\frac{n^2}{2(n-1)}f^2] \geqslant 0.
\end{eqnarray}

So from the proof of Lemma 3.1 and Theorem 1.1, given $\int_M |\phi|^{2\gamma} dv<\infty$ for a certain $\gamma\geqslant \frac{1}{2}-\frac{1}{n}$, we obtain
\begin{equation}
|\phi|^{2\gamma}(x_0)\leqslant \frac{C(n,H)}{\eta^{n-\delta}}2^{(2+2\delta)\gamma}\int_{D_{\eta}(x_0)}|\phi|^{2\gamma}dv\leqslant \widetilde{C}(n,H)2^{(2+2\delta)\gamma}\eta^{\delta-n}.
\end{equation}
By choosing $\delta>\frac{n-2\gamma}{2\gamma+1}$, we can actually prove a pointwise estimate similar to Theorem 1.1, which means $|\phi|$ is bounded on $M$, given $\partial M=\emptyset$.

\begin{bound}
Let $M^n$ be an oriented hypersurface with constant mean curvature immersed in $\mathbb{R}^{n+1}$ $($$n\geqslant 2$$)$ and $\partial M=\emptyset$. If there exists
$p\geqslant 1-\frac{2}{n}$, such that $$\int_M |\phi|^p dv<\infty,$$
then $|\phi|$ is bounded on $M$.
\end{bound}

More precisely, we have the following theorem.

\begin{pro1}
Suppose $M^n$ is a complete oriented non-compact CMC hypersurface immersed in $\mathbb{R}^{n+1}$ $($$n\geqslant 2$$)$ with $\partial M=\emptyset$. If there exists $p\geqslant 1-\frac{2}{n}$,
such that $$\int_M |\phi|^p dv<\infty.$$
Then for any $x_0\in M$ and $\varepsilon >0$, there exists $R(x_0)>0$, such that $$\sup_{M\setminus B_{R}(x_0)}|\phi|<\varepsilon.$$
\end{pro1}

\begin{proof}
From the finite $L^p$-norm curvature condition, we know from Proposition 4.1 that $|\phi|$ is bounded on $M$. Fix $x_0\in M$. Then for any $\varepsilon >0$, there exists $R(x_0)>0$, such that $$\int_{M\setminus B_{R}(x_0)}|\phi|^p<\varepsilon.$$
Similar to the proof of Theorem 3.3, by taking $R_0=1$, we have
$$\sup_{M\setminus B_{R+1}(x_0)}|\phi|<C(n,H)^{\frac{1}{4\gamma}}\varepsilon^{\frac{1}{4\gamma}},$$
which finishes the proof when $p\geqslant 2-\frac{4}{n}$.\\
If $1-\frac{2}{n}\leqslant p<2-\frac{4}{n}$, the conclusion is still true, since $|\phi|$ is bounded.
\end{proof}

By Theorem 4.2, we are able to prove a compactness theorem. See \cite{DoCS} for reference.
\begin{pro2}
Suppose $M^n$ is a complete oriented CMC hypersurface immersed in $\mathbb{R}^{n+1}$ $($$n\geqslant 2$$)$ with $H\neq 0$ and $\partial M=\emptyset$. If there exists $p\geqslant 1-\frac{2}{n}$,
such that $$\int_M |\phi|^p dv<\infty.$$
Then $M$ must be compact.
\end{pro2}

\indent Now we begin the proof of Theorem 1.3.
\begin{proof}[Proof of Theorem 1.3]
Case 1: If $H\neq 0$, Theorem 4.3 shows that $M$ is compact. Then $M$ has to be a round sphere by \cite{BACA} if $M$ is weakly stable.\\
\indent Case 2: If $M$ is non-compact, then $M$ must be minimal also by Theorem 4.3.\\
\indent Case 3: If $M$ is stable minimal, and there exists $1-\frac{2}{n} \leqslant p \leqslant n$, such that $\int_M |A|^p dv<\infty$. Then $M$ has to be a hyperplane by Shen and Zhu's result in \cite{ShZh}, since $|A|$ is bounded on $M$ by Proposition 4.1.
\end{proof}

\begin{4}
Compared to present results of this kind, there are little restrictions to $n$ and $p$, but we need the condition $\partial M=\emptyset$ in our proof, which may not be necessary. The proof of Theorem $1.3$ is still true for hypersurfaces mentioned in Remark $2$, whose boundary point is not empty, but we can not handle manifolds like the open upper half plane, which is obviously not complete. So it is natural to ask if there is any complete non-compact hypersurface isometrically immersed in $\mathbb{R}^{n+1}$ $($$n\geqslant 2$$)$ with $\partial M\neq\emptyset$, satisfying that there exists universal $R_M$, such that $D_{R_M}(x)$ is connected for any $x\in M$.
\end{4}

\end{document}